\documentclass[13pt]{amsart}
    \usepackage{amsmath}
    \usepackage{}
    \usepackage{amsfonts}
    \usepackage{amsfonts, latexsym, rawfonts, amsmath, amssymb, amsthm}
    \usepackage[plainpages=false]{hyperref}

    \usepackage{pdfsync}

    \usepackage[a4paper]{geometry}
    \usepackage{esint}
    \usepackage{graphics, graphicx}
    \usepackage{tikz}
    \usepackage{appendix}

    \numberwithin{equation}{section}
    \newcommand{\beq}{\begin{equation}}
    \newcommand{\eeq}{\end{equation}}
    \newcommand{\beqs}{\begin{eqnarray*}}
    \newcommand{\eeqs}{\end{eqnarray*}}
    \newcommand{\beqn}{\begin{eqnarray}}
    \newcommand{\eeqn}{\end{eqnarray}}
    \newcommand{\beqa}{\begin{array}}
    \newcommand{\eeqa}{\end{array}}

    \def\lra{\longrightarrow}

    \def\bc{\begin{center}}
    \def\ec{\end{center}}

    \def\begeq{\begin{equation}}
    \def\endeq{\end{equation}}
    \def\and{\quad{\rm and}\quad}

    \let\lra=\longrightarrow

    \def\mapright\#1{\, \smash{\mathop{\lra}\limits^{\#1}}\, }

    \newtheorem{prop}{Proposition}[section]
    \newtheorem{theo}[prop]{Theorem}
    
    \newtheorem{claim}[prop]{Claim}
    
    \newtheorem{rem}[prop]{Remark}
    
    \newtheorem{defi}[prop]{Definition}

\begin{document}

     \title{Lower bound of modified $K$-energy on a Fano manifold with degeneration for K\"ahler-Ricci solitons}

% An explicit formula  for the   second variation of Perelman's entropy
    %\date{-}

    \author{Liang  Zhang}

    \address{School of Mathematical Sciences, Peking
    University,  Beijing 100871,  China.}

    \email{ tensor@pku.edu.cn}
 %\thanks {$\dag$  Partially supported by NSFC Grants 11971423}
    %\thanks {* Partially supported by NSFC Grants 11771019 and BJSF Grants Z180004.}

\begin{abstract}
In this paper, we extend Tosatti's method to study the lower boundedness of modified $K$-energy on a Fano manifold and apply this result to study the relative $K$-stability of the deformation space of a K\"ahler Ricci soliton.

%We generalize a result of X.X. Chen to the K\"ahler Ricci soliton case by showing the lower bound of modified K energy. And this generalization implies that small deformation of K\"ahler Ricci soliton where soliton vector field can be lifted is K-semistable.
\end{abstract}

\date{}

\maketitle

\tableofcontents

\setcounter{section}{-1}

\section{Introduction}
Let $(M,J)$ be a Fano manifold with soliton vector field $X$. By the virtue of Yau-Tian-Donaldson conjecture, the study of K\"ahler Ricci soliton is related to the notion of $K$-stability for $(M,X)$. For example, it is well known that the existence is equivalent to the $K$-polystability (see \cite{BN} or\cite{Sz}). We are interested to establish the semistable version of Yau-Tian-Donaldson correspondence. %We want to show that the lower boundedness of modified K energy is equivalent to the K-semistability.

When $X=0$, Li \cite{Li17} solved this problem by showing a lot of equivalent characterization of $K$-semistability. The most important contribution of his proof is the implication from $K$-semistablity to the lower boundedness of $K$-energy. This is a generalization of the result of Chen \cite{CS08} and  Tosatti \cite{To12} who derived the lower boundedness under the assumption that $M$ admits a smooth degeneration with K\"ahler Einstein metric.
%the boundedness of K energy was derived by Chen under the assumption that $(M,J)$ can degenerate to a smooth K\"ahler Einstein manifold\cite{CS08}. We can also get the boundedness of K energy if $(M,J)$ admits a sequence of K\"ahler metrics converging to a K\"ahler Einstein manifold in $C^{\infty}$ topology\cite{ZhT}. Recently, Li has given a proof of the equivalence of K-semistability and lower boundedness of K-energy even for singular degeneration \cite{Li17}.

However, for the nontrivial soliton case, it seems that the implication remains unknown. Fortunately, we still know that the $K$-semistabily is equivalent to the existence of $K$-polystable degeneration \cite{Li20}. Thus this problem can be reduced to researching whether the existence of polystable degeneration implies the lower boundedness of the modified $K$-energy. The main purpose of this paper is to derive the implication under the assumption that the polystable degeneration is smooth.

Our method is a generalization of Tosatti's proof \cite{To12} for the K\"ahler Einstein case. The key technique of his proof is a slope-type inequality about the $K$-energy, which was discovered by Chen \cite{CS09}. This inequality was proved by many different methods (see also \cite{CS12}) and had also been used to prove the lower boundedness of K energy along Calabi flows \cite{CS14}.

Note that this slope-type inequality can be generated for the modified $K$-energy (and other energy in more general situations \cite{AB}).
We will prove the following theorem in Section 2:
 %modified K-energy $\mu_{\omega}$ in \cite{TZ02} for an extremal vector field $X$ which is a holomorphic vector field such that the modified Futaki invariants is vanishing on $\eta_r(M,J)$, the reductive Lie subgroup of holomorohic vector field. and prove that if a Fano manifold $M$ admits a K\"ahler Ricci soliton $(\omega_{KS},X)$ then the modified K-energy related to the holomorphic vector field is bounded from below. The boundedness of modified $K$-energy has a lot of applications. For example, in \cite{TZZZ} Tian-Zhang-Zhang-Zhu use it to study the energy level.

%Later on, Berman-Nystrom \cite{BN} generalize the notion of $K$-stability to the K\"ahler Ricci soliton case and prove that the existence of K\"ahler Ricci soliton implies $K$-Polystability. And their proof also shows that the boundedness of modified $K$-energy implies $K$-semistability.

%On the other hand, there are a lot of work of studying the lower bound of $K$-energy. In \cite{CS08}, Chen study the lower bound of $K$-energy along geodesics. He prove that if the manifold admits a smooth test configuration with central fiber admits a csck metric, the K energy is bounded from below.  This result is reproved by Tosatti \cite{To12} by using a slope-type inequality of $K$-energy which is proved by Chen in \cite{CS09}.

%In this paper, we will extend the proof of Tosatti to the soliton case and derive the following theorem which is a generalization of Chen's result \cite{CS08}.

\begin{theo}\label{main-theorem}
Let $\pi:\mathcal{M}\mapsto \mathbb{C}$ be a smooth special degeneration associated to the soliton action induced by $X$. Suppose that there is a $T\times S^1$ invariant K\"ahler metric near central fiber (c.f. Section 1) and the central fiber $M_0$ admits a K\"ahler Ricci soliton. Then the modified $K$-energy on $M$ is bounded from below.
\end{theo}

After establishing this theorem, we can apply it to study the deformation space of Eiji Inoue \cite{Ino19}, which is the same as the definition of the kernel space of second order variation of Perelman's entropy \cite{TZZ}.  We will prove:
%Combine this Theorem and the deformation theory of Eiji Inoue \cite{Ino19}, we will get the main theorem of our paper which generalize the result of Tosatti \cite{To12}:

\begin{theo}\label{semistable}
Let $(M,J_0)$ be a Fano manifold which admits a K\"ahler Ricci soliton, $(M,J)$ be a sufficiently small deformation of $(M,J_0)$. Suppose that  the soliton vector field on $(M,J_0)$ can be lifted to $(M,J)$. Then the modified K energy on $(M,J)$ is bounded from below.
\end{theo}

We obtain a family of smooth manifolds on each of which the modified K energy is bounded from below. Furthermore, the lower boundedness of modified $K$-energy for $(M,X)$ implies that the energy level of $M$ satisfies
\begin{align}\label{entropy}
\sup_{\omega_g\in 2\pi c_1(M)}\lambda(g)=(2\pi)^{-n}(nV-N_X(c_1(M)))
\end{align}
(see \cite{DeS} and \cite{WZ1}). Thus we derive that the energy level of manifold in Theorem \ref{semistable} is independent of the complex structure, which has been observed in \cite{TZZ} by the method of K\"ahler Ricci flow.
%\begin{rem}By a result of  Dervan-Sz\'ekelyhidi  \cite{DeS} (also see \cite{WZ1}), we know that
%\begin{align}\label{maximum}\sup_{\omega_g\in 2\pi c_1(M, J_\varphi)}\lambda(g)=\lambda(g_{KS}).
%\end{align}

%Hence by \cite{BN} we conclude that small deformations of $(M,J_0)$ such that the soliton vector field on $(M,J_0)$ can be lifted to $(M,J)$ are $K$-semistable.
This paper is organized as follows:

In Section 1, we recall the notion of  special degeneration and study some basic setups. In Section 2, we prove Theorem \ref{main-theorem}. Finally in Section 3, we prove Theorem \ref{semistable} by showing that the manifold appearing in Theorem \ref{semistable} admits a smooth special degeneration.
\section{Preliminary}
In this section, we recall the notion of  special degeneration and study some basic setups.
Let $M$ be a Fano manifold with $X$ being a soliton vector field on $M$.

Recall that   a  special degeneration of a  Fano manifold $M$ is a  normal variety $\mathcal M$ with a $\mathbb C^*$-action satisfying the follow conditions \cite{Ti97}:
\begin{enumerate}
\item There exists a flat $\mathbb C^*$-equivarant map $\pi: \mathcal M\to \mathbb C$ such that $\pi^{-1}(t)$ is biholomorphic to $M$ for any $t\neq 0$;

\item There exists an holomorphic line bundle $\mathcal L$ on $\mathcal M$ such that for any $t\neq 0$, $\mathcal L|_{\pi^{-1}(t)}$ is isomorphic to $K_M^{-r}$ for some integer $r>0$;

    \item The center $M_0= \pi^{-1}(t)$ which is a $Q$-Fano variety.
\end{enumerate}
%Denote the action of $\mathbb{C}^*$ on $\mathcal{X}$ by the map
%$$\rho:\mathbb{C}^*\mapsto {\rm Aut}(\mathcal{X})$$
%For $z\in \mathbb{C}$, let $\mathcal{X}_{z}=\pi^{-1}z$, we define the map $F_{z}:M\mapsto \mathcal{X}_{z}$ by
%$$F_{z}(m)=\theta(m,z).$$
%$F_{z}$ is a family of isomorphism of complex manifold. Suppose that $\Omega$ is a $T\times S^1$ invariant K\"ahler metric on $\mathcal{X}$. Let $V$ be the real vector field on $\mathcal{X}$ which generated the action of $S$ on $\mathcal{X}$. Thus we have
%$$\mathcal{L}_{V}\Omega=di_{V}\Omega=0.$$
%Since $H^1(\mathcal{X},\mathbb{R})=0$, we may find a smooth function $H_{V}$ on $\mathcal{X}$ such that
%$$i_{V}\Omega=dH_{V}.$$
%Similarly, let $W$ be the real vector field on $\mathcal{X}$ which generated the action of $T$ on $\mathcal{X}$, we may find smooth function $H_{W}$ such that
%$$i_{W}\Omega=dH_{W}.$$
%Let $i_z:\mathcal{X}_{z}\hookrightarrow \mathcal{X}$ be the inclusion map. Let
The following definition can be seen in \cite{WZZ16}.

\begin{defi}
 $\mathcal{M}$ is called a special degeneration associated to the soliton action induced by $X$ if $\sigma_t^{v}$ communicates to $\sigma_t^X$, where $\sigma_t^{X}$ and $\sigma_t^v$ are two lifting one-parameter subgroups on $\mathcal{M}$ induced by $X$ and the holomorphic vector field $v$ associated to the $\mathbb{C}^*$ action, respectively.
\end{defi}

If $M_0$ is smooth and there exists an neighborhood $\Delta=\{|z|<\epsilon\}$ such that $\pi^{-1}(\Delta)$ admits a $T\times S^1$ invariant K\"ahler metric $\Omega$. We call $\mathcal{M}$ a smooth special degeneration with a $T\times S^1$ invariant K\"ahler metric near central fiber. Here $T$ and $S$ are one-parameter subgroups on $\mathcal{M}$ induced by $\xi={\rm Im}(X)$ and ${\rm Im}(v)$, respectively.

Since $M_0$ is smooth, we know that $\mathcal{M}$ is smooth and $\pi$ is holomorphic proper submersion. By Ehresmann's theorem, we can find a neighborhood $\Delta=\{|z|<\epsilon\}$ of 0 and a diffeomorphism
\begin{align}\label{diffeomorphism}
 F:\underline{M}\times \Delta\mapsto \pi^{-1}(\Delta)
\end{align}
such that $\pi(F(m,z))=z$.  Here we use $\underline{M}$ to denote the underlying differential manifold of $(M, J)$.

%Without loss of generality, we assume $1\in \Delta$.
By the definition of $\mathcal{M}$, there is a $T\times \mathbb{C}^*$ action on $\mathcal{M}$ such that $\pi$ is $T\times \mathbb{C}^*$ equivalent. We may induce a local action of $T\times \mathbb{C}^*$ on $\underline{M}\times \Delta$ by $F$, which satisfying:
\begin{align}
(w,s)\cdot (m,z)=F^{-1}((w,s)\cdot F(m,z)),
\end{align}
if $sz\in \Delta$. %Let $\mathcal{J}=F^*J_{\mathcal{M}}$ be the complex structure. Here $J_{\mathcal{M}}$ is the complex structure of $\mathcal{M}$.
Note that $T\times S^1$ maps $\underline{M}\times \Delta$ to itself. Hence this local action forces $\underline{M}\times \Delta$ to admit a $T\times S^1$ action.

We can also induce a K\"ahler metric on $\underline{M}\times \Delta$ through $F$. Since $F$ is $T\times S^1$ equivalent, this metric is also  $T\times S^1$ invariant. We still denote by $\Omega$. Let $V$ be the real vector field on $\underline{M}\times \Delta$ which generates the action of $S^1$ on $\underline{M}\times \Delta$. Thus we have
\begin{align}\label{invariant-metric}
 \mathcal{L}_{V}\Omega=d\iota_{V}\Omega=0.
\end{align}
Since $H^1(\underline{M}\times \Delta,\mathbb{R})=0$, we may find a smooth function $H_{V}$ on $\underline{M}\times \Delta$ such that
\begin{align}
\iota_{V}\Omega=dH_{V}.
\end{align}
Similarly, let $W$ be the real vector field on $\underline{M}\times \Delta$ which generates the action of $T$ on $\underline{M}\times \Delta$, and we may find a smooth function $H_{W}$ such that
\begin{align}\label{soliton-potential-global}
  \iota_{W}\Omega=dH_{W}.
\end{align}

Let $\mathcal{J}=F^*J_{\mathcal{M}}$ be the complex structure induced by $F$. Here $J_{\mathcal{M}}$ is the complex structure of $\mathcal{M}$. It is easy to see that $\sqrt{-1}W+\mathcal{J}W$ tangents to each fiber $M_z$ and it's restriction $X_z=\sqrt{-1}W|_{M_z}+\mathcal{J}|_{M_z}W|_{M_z}$ is the soliton vector field on $M_z$. By restricting (\ref{soliton-potential-global}) %to $M_z$, we have
%\begin{align}\label{soliton-potential}
%  i_{W|_{M_z}}\Omega|_{M_z}=d(H_{W}|_{M_z}).
%\end{align}
we see that the soliton potential of $X_z$ on $M_z$ respect to $\Omega|_{M_z}$ is $H_{W}|_{M_z}$.

%It is easy to see that the action of $T$ preserves the fiber $M_{z}=\underline{M}\times \{z\}$. Meanwhile, we know that the action of $T$ on $M_z$ is generated by the imaginary of its soliton vector field. Then $\widetilde{X}_z=\sqrt{-1}W|_{M_z}+\mathcal{J}|_{M_z}W|_{M_z}$ is the soliton vector field on $M_z$. By restricting (\ref{soliton-potential-global}) to $M_z$, we have
%\begin{align}\label{soliton-potential}
%  i_{W|_{M_z}}\Omega|_{M_z}=d(H_{W}|_{M_z}).
%\end{align}
%So we conclude that the soliton potential of $\widetilde{X}_z$ on $M_z$ respect to $\Omega|_{M_z}$ is $H_{W}|_{M_z}$.

In addition, we may construct a family of metric on $M$ by using the action of $\mathbb{C}^*$. Let
\begin{align}
F_{t}:\underline{M}\times \Delta\mapsto \underline{M}\times \Delta, F_{t}(m,z)=e^{-t}\cdot(m,z), t>0
\end{align}
and $f_t=F_t\circ i$, where $i:M\mapsto \underline{M}\times \Delta, i(m)=(m,1)$. We can define
\begin{align}
 \omega_t=f_t^*\Omega
\end{align}
as a family of K\"ahler metric on $M$. We will show that this family decay fast in some sense.

Let $\rho_t: M_{e^{-t}}\mapsto M$ be the inverse of $f_t:M\mapsto f_t(M)$. Note that $\rho_t^*\omega_t=\Omega|_{M_{e^{-t}}}$. We conclude that
\begin{align}\label{kahler-form}
  \|\rho_t^*\omega_t-\Omega|_{M_0}\|_{g}\leq Ce^{-t}.
\end{align}
Here $g$ is a fixed Riemmannian metric on $\underline{M}$.% Let $G$ be the metric on $\underline{M}\times \Delta$ defined by $\Omega$ and $\mathcal{J}$. $G|_{M_{e^{-t}}}$ is the Riemmannian metric on $M_{e^{-t}}$ defined by $\mathcal{J}|_{M_{e^{-t}}}$ and $\Omega|_{M_{e^{-t}}}$. Let $g_{t}$ be the metric on $M$ defined by $\omega_t$. Then we have $\rho_t^*g_t=G|_{M_{e^{-t}}}$. Therefore, we also have
%\begin{align}\label{Riemann-metric}
%  \|\rho_t^*g_t-G|_{M_0}\|_{g}\leq Ce^{-t}.
%\end{align}

In addition, we may write $\omega_t$ as $\omega_t=\omega_0+dd^c\varphi_t$. Since
\begin{align}
  \frac{d}{dt}\omega_t=dd^cf_t^*H_V.
\end{align}
We may assume that $\dot{\varphi}_t=f_t^*H_V$. As a result, we have that
\begin{align}\label{potential}
  \|\rho_t^*\dot{\varphi}_t-H_V|_{M_0}\|_{g}\leq Ce^{-t}.
\end{align}

Finally, since the isomorphism $f_t$ pulls back the soliton vector field $X_{e^{-t}}$ on $M_{e^{-t}}$ to $X$, we conclude that the soliton potential $\theta_t=\theta_{X}(\omega_t)$ of $X$ respect to $\omega_t$  is $f_t^*H_W$. Consequently, we have
\begin{align}\label{soliton-potential-estimate}
  \|\rho_t^*\theta_t-H_W|_{M_0}\|_{g}\leq Ce^{-t}.
\end{align}

\section{Proof of Theorem 0.1}
In this section we prove the Theorem \ref{main-theorem}.

\begin{proof}[Proof of Theorem 0.1]
  Let $\Omega$ be a $T\times S^1$ invariant  K\"ahler metric on $\underline{M}\times \Delta$.

\begin{claim}\label{invariant-metric-global}
  We may assume that $\Omega|_{M_0}$ is the soliton metric of $M_0$ respect to soliton vector field $X_0$.
\end{claim}
Let $\omega_t$ be the family of metric on $M$ defined in Section 1 and $\omega=\omega_0$. We will prove that $\mu_{\omega}$ is bounded from below.

 Let $\varphi\in \mathcal{M}_{\omega}$, where
 \begin{align}
\mathcal{M}_{\omega}=\{\varphi\in C^{\infty}(M)|\omega+dd^c\varphi>0, {\rm Im}(X)(\varphi)=0\}.
 \end{align}
 We may choose a path $\varphi_{t}, t\in [-1,0]$ such that $\varphi_{-1}=\varphi$ and $\varphi_0=0$. Connecting it with $\varphi_t, t\geq0$ we get a ray $\{\varphi_t:t\geq -1\}$. Then for $t>0$, the derivative of $\mu_{\omega}(\varphi_t)$ is
 \begin{align}
   \frac{d}{dt}\mu_{\omega}(\varphi_t)&=-\int_M\dot{\varphi}(t)(\Delta_{g_t}+X)(h_{\omega_t}-\theta_t)\omega_t^n\notag\\
   &=-\int_{\underline{M}}\rho_t^*\dot{\varphi}(t)(\Delta_{\rho_t^*g_t}+X_{e^{-t}})(h_{\rho_t^*\omega_t}-\rho_t^*\theta_t)\rho_t^*\omega_t^n.
 \end{align}
Since $\Delta_{\rho_t^*g_t}h_{\rho_t^*\omega_t}=R(\rho_t^*g_t)-n$, by (\ref{kahler-form}) we see that
 \begin{align}\label{Ricci-potential}
   \|h_{\rho_t^*\omega_t}-h_{\Omega|_{M_0}}\|\leq Ce^{-t}.
 \end{align}
It follows from (\ref{soliton-potential-estimate}) and (\ref{Ricci-potential}) that
\begin{align}
  \|h_{\rho_t^*\omega_t}-\rho_t^*\theta_t-(h_{\Omega|_{M_0}}-H_W|_{M_0})\|\leq Ce^{-t}.
\end{align}
Note that $\Omega|_{M_0}$ is a soliton metric and $H_W|_{M_0}$ is soliton potential. We have
\begin{align}\label{central-equation}
  h_{\Omega|_{M_0}}=H_W|_{M_0}.
\end{align}
It follows that
\begin{align}\label{estimate}
  \|h_{\rho_t^*\omega_t}-\rho_t^*\theta_t\|\leq Ce^{-t}.
\end{align}
Meanwhile, by (\ref{potential}) and (\ref{kahler-form}) and the fact that
\begin{align}
\|X_{e^{-t}}-X_0\|\leq Ce^{-t},
\end{align}
we derive that $\frac{d}{dt}\mu_{\omega}(\varphi_t)$ converges exponentially to
\begin{align}
  -\int_{\underline{M}}H_{V}(\Delta_{G|_{M_0}}+X_0)(h_{\Omega|_{M_0}}-H_W|_{M_0})(\Omega|_{M_0})^n=0.
\end{align}
As a result, we have
\begin{align}\label{modified-K-energy}
  \mu_{\omega}(\varphi_t)\geq -C.
\end{align}

Furthermore, we have the Chen inequality (see Corollary 1 in \cite{AB}) for modified $K$-energy
\begin{align}\label{slope inequality}
  \mu_{\omega}(\varphi_{-1})\geq\mu_{\omega}(\varphi_t)-d(\varphi_{-1},\varphi_{t})\sqrt{\widetilde{Ca}(\omega_t)}.
\end{align}
 Here
 \begin{align}
 d(\varphi_{-1},\varphi_{t})=\int_{-1}^t\sqrt{\int_M (\dot{\varphi}(s))^2\omega_s^n}ds
 \end{align}
and
\begin{align}
 \widetilde{Ca}(\omega_t)=\int_{M}[(\Delta_{g_t}+X)(h_{\omega_t}-\theta_t)]^2e^{2\theta_t}\omega_t^n.
\end{align}
By (\ref{kahler-form}) and (\ref{estimate}), we conclude that
 \begin{align}
   |\widetilde{Ca}(\omega_t)-\int_M[(\Delta_{G|_{M_0}}+X_0)(h_{\Omega|_{M_0}}-H_W|_{M_0})]^2e^{2H_W|_{M_0}}(\Omega|_{M_0})^n|\leq Ce^{-2t}.
 \end{align}
 Hence by (\ref{central-equation}), it follows that
 \begin{align}\label{calabi-energy}
   \widetilde{Ca}(\omega_t)\leq Ce^{-2t}.
 \end{align}

Finally, we see that for $s>0$,
 \begin{align}
   \int_M (\dot{\varphi}(s))^2\omega_s^n&=\int_M(\rho_t^*\dot{\varphi}(s))^2\rho_t^*\omega_s^n.
 \end{align}
It follows from (\ref{kahler-form}) and (\ref{potential}) that $\int_M (\dot{\varphi}(s))^2\omega_s^n$ is uniformly bounded for $s>0$. Thus we have
\begin{align}\label{distance}
  d(\varphi_{-1},\varphi_{t})=\int_{-1}^0\sqrt{\int_M (\dot{\varphi}(s))^2\omega_s^n}ds+\int_{0}^t\sqrt{\int_M (\dot{\varphi}(s))^2\omega_s^n}ds\leq Ct+D.
\end{align}
Combining (\ref{modified-K-energy}), (\ref{slope inequality}), (\ref{calabi-energy}) and (\ref{distance}), we conclude that
\begin{align}
 \mu_{\omega}(\varphi)=\mu_{\omega}(\varphi_{-1})\geq-C.
\end{align}
Thus $\mu_{\omega}$ is bounded from below. We finish the proof.
 \end{proof}

To complete the proof, we prove Claim (\ref{invariant-metric-global}) as following:
\begin{proof}[Proof of Claim (\ref{invariant-metric-global})]
Since we assume that $M_0$ admits a Kahler Ricci soliton, and the action of $S^1$ commutes with the action of $T$ on $M_0$, we may find a $T\times S^1$ invariant function $\widehat{\psi}$ on $M_0$ such that $\Omega|_{M_0}+dd^c\widehat{\psi}$ is the soliton metric of $M_0$ with respect to soliton vector field $X_0$. As $T\times S^1$ is compact, we can extend $\widehat{\psi}$ to be a $T\times S^1$ invariant smooth function on $\underline{M}\times \Delta$. We denote it by $\psi$. Shrinking $\Delta$ if it is necessary, we may assume that $\Omega+dd^c\psi+add^c|z|^2$ is a $T\times S^1$ invariant K\"ahler metric on $\underline{M}\times \Delta$ such that  $(\Omega+dd^c\psi+add^c|z|^2)|_{M_0}$ is a soliton metric of $M_0$ respect to soliton vector field $X_0$. Here $a>0$ is a big positive number. As a result, replacing $\Omega$ by $\Omega+dd^c\psi+add^c|z|^2$, we conclude that Claim (\ref{invariant-metric-global}) is true.
\end{proof}

\section{Proof of the Theorem \ref{semistable}}
In this section we prove the Theorem \ref{semistable}.

\begin{proof}[Proof of the Theorem \ref{semistable}]
  First at all, we may construct a smooth special degeneration associated to the soliton action on $(M,J)$. We refer the readers to the proof of Theorem 0.2 in \cite{TZZ} for the details. Since the soliton vector field of $(M,J_0)$ can be lifted to $(M,J)$, we know that the K\"ahler Ricci flow $(M,g(t))$ on $(M,J)$ converges smoothly to a K\"ahler Ricci soliton $(M_{\infty},J_{\infty},g_{\infty})$ by the Theorem 0.1 in that paper. Then we can embed $(M,g(t))$ to a  projective space $\mathbb{P}^N$ by partial $C^0$-estimate for $t\geq t_0$ with $\sigma_s^X$ being regarded as a subgroup of $SL(N+1,\mathbb{C})$. By GIT, we will find a fixed number $t_1\geq t_0$, and a one parameter subgroup $\sigma_{t}\subseteq SL(N+1,\mathbb{C})$ which commutes with $\sigma_s^X$  such that $\sigma_{t}(\widetilde{M}_{t_1})$ converges to a limit cycle $\widetilde{M}_{\infty}$ which is isomorphic to $(M_{\infty},J_{\infty})$. Hence we can construct a special degeneration $\mathcal{M}\subset \mathbb{P}^N\times \mathbb{C}$ as the compactification of
  \begin{align}
  S=\{(x,t)\in \mathbb{P}^N\times \mathbb{C} |x\in \sigma_{t}(\widetilde{M}_{t_1})\},
  \end{align}
  whose central fiber is $\widetilde{M}_{\infty}$\cite{Mum77}.
  There is a nature way to introduce the action of $\mathbb{C}^*\times \mathbb{C}^*$ on $\mathbb{P}^N\times \mathbb{C}$ as
  \begin{align}
 (t,s)(x,z)=(\sigma_t\sigma_s^X(x),tz),~(t,s)\in \mathbb{C}^*\times \mathbb{C}^*,~(x,z)\in\mathbb{P}^N\times \mathbb{C}.
  \end{align}
  Note that $S$ is invariant under the action of $\mathbb{C}^*\times \mathbb{C}^*$. We know that $\mathcal{M}$ is also invariant. Thus $\mathcal{M}$ is a special degeneration of $(M,J)$ associated to the soliton action. Since the central fiber is $\widetilde{M}_{\infty}$ and this family is flat, we conclude that it is also a smooth special degeneration and $\mathcal{M}$ is a smooth submanifold of $\mathbb{P}^N\times \mathbb{C}$ (see proposition 10.2 in \cite{GTM52}).

  Secondly, as the compact subgroup of $\sigma_t$ commutes with the compact subgroup of  $\sigma_s^X$, we may find a K\"ahler metric $\omega$ of $\mathbb{P}^N$ such that $\omega$ is invariant under the action of these two compact subgroups. Therefore, we can construct a $T\times S^1(\subseteq \mathbb{C}^*\times \mathbb{C}^*)$ invariant metric $\mathbb{P}^N\times \mathbb{C}$ as
  \begin{align}
  \Omega=\omega+\sqrt{-1}dz\wedge d\overline{z}.
  \end{align}
  Restricting $\Omega$ to the $\mathcal{M}$, we derive a $T\times S^1$ invariant metric of $\mathcal{M}$.

  Finally, we can apply the Theorem \ref{main-theorem} to finish the proof of Theorem \ref{semistable}.
\end{proof}

\begin{rem}
  We have shown that for  K\"ahler Ricci soliton $(M,J,\omega_{FS})$, the soliton metric $\omega_{FS}$ can be viewed as a K\"ahler  metric on each manifold appearing in the deformation family of it \cite{TZZ}. So we can construct a $K$ invariant K\"ahler metric on the deformation space. Here $K$ is a maximal compact subgroup of ${\rm Aut}_r(M,J)$ respect to $\omega_{FS}$. Hence, by GIT and Eiji Inoue's deformation Theorem \cite{Ino19} we may construct a smooth degeneration with $T\times S^1$ invariant K\"ahler metric near central fiber for each manifold appearing in this family. As a result,  we can also prove Theorem \ref{semistable} by Theorem \ref{main-theorem} and this construction.
\end{rem}

   \end{document}